\documentclass[12pt,reqno]{amsart}
\usepackage{palatino}
\usepackage{amsfonts,amssymb,latexsym,amsmath, amsxtra, mathrsfs}
\usepackage{verbatim}
\usepackage{diagbox}

\theoremstyle{plain}
\newtheorem{theorem}{Theorem}[section]
\newtheorem{corollary}[theorem]{Corollary}

\newtheorem{conjecture}[theorem]{Conjecture}

\newtheorem{proposition}[theorem]{Proposition}

\newtheorem*{theorem*}{Theorem}
\theoremstyle{definition}

\theoremstyle{remark}
\newtheorem*{remark}{Remark}
\newtheorem*{remarks}{Remarks}

\numberwithin{equation}{section}

\newcommand{\R}{\mathbb R}
\newcommand{\N}{\mathbb N}
\newcommand{\Z}{\mathbb Z}
\newcommand{\C}{\mathbb C}

\newcommand{\Q}{{\mathbb Q}}

\def\sS{\mathcal S}

\def\rank{{\rm rank}}
\def\crank{{\rm crank}}

\def\({\left(}
\def\){\right)}
\def\<{\left<}
\def\>{\right>}

\newcommand{\SL}{\text{SL}}

\newcommand{\abs}[1]{\left|#1\right|}

\newcommand{\Log}{{\rm Log}}

\def\cU{\mathcal U}

\def\cQ{\mathcal Q}

\def\rk{{\rm rank}}

\def\ospt{{\rm ospt}}

\renewcommand{\b}[1]{\boldsymbol{#1}}

\renewcommand{\pmod}[1]{\  \,  \left(  \mathrm{mod} \,  #1 \right)}

\begin{document}
\allowdisplaybreaks

\title[Peak Positions of Unimodal Sequences]{Peak positions of  strongly unimodal sequences}
\author[K. Bringmann]{Kathrin Bringmann}
\author[C. Jennings-Shaffer]{Chris  Jennings-Shaffer}
\author[K. Mahlburg]{Karl Mahlburg}
\author[R. Rhoades]{Robert Rhoades}

\address{Mathematical Institute, University of Cologne, Weyertal 86-90, 50931 Cologne, Germany}
\email{kbringma@math.uni-koeln.de}

\address{Mathematical Institute, University of Cologne, Weyertal 86-90, 50931 Cologne, Germany}
\email{cjenning@math.uni-koeln.de}

\address{Department of Mathematics, Louisiana State University, Baton Rouge, LA 70803, USA}
\email{mahlburg@math.lsu.edu}

\address{Susquehanna International Group, Bala Cynwyd, PA 19004, USA}
\email{rob.rhoades@gmail.com}

\thanks{
	The research of the  first author is supported by the Alfried Krupp Prize for Young University Teachers of the Krupp foundation and the research leading to these results receives funding from the European Research Council under the European Union's Seventh Framework Programme (FP/2007-2013) / ERC Grant agreement n. 335220 - AQSER}
\date{\today}
\thispagestyle{empty} \vspace{.5cm}
\begin{abstract}
We study combinatorial and asymptotic properties of the rank of strongly unimodal sequences. We find a generating function for the rank enumeration function, and give a new combinatorial interpretation of the ospt-function introduced by Andrews, Chan, and Kim. We conjecture that the enumeration function for the number of unimodal sequences of a fixed size and varying rank is log-concave, and prove an asymptotic result in support of this conjecture. Finally, we determine the asymptotic behavior of the rank for strongly unimodal sequences, and prove that its values (when appropriately renormalized) are normally distributed with mean zero in the asymptotic limit.

\end{abstract}



\maketitle


\section{Introduction and Statement of Results}\label{sec:intro}

The study of combinatorial statistics for integer partitions has led to a vast number of interesting results, including bijective correspondences, congruences, asymptotic formulas and inequalities, limiting probabilistic distributions, and striking examples of modular forms
(and generalizations thereof). In this paper we consider many of
these questions for unimodal sequences, which have only recently been considered from a number-theoretic perspective.

\subsection{History of statistics for integer partitions}

We begin by recalling standard definitions from the theory of partitions \cite{And98}. 
A sequence of positive integers $\left( \lambda_j\right)_{j=1}^\ell$
is a {\it partition of size $n$} if it satisfies
\begin{equation*}
\lambda_1\ge \lambda_2 \ge \dotsc \ge \lambda_\ell,
\qquad\mbox{and}\quad
\lambda_1+\dots +\lambda_\ell=n.
\end{equation*}
If $\lambda$ is a partition, then we denote its size by $|\lambda|$, and its length by $\ell(\lambda) = \ell$.
We define $p(n)$ as the number of partitions of $n$.
For example, the partitions of $5$ are
$(5)$, $(4,1)$, $(3,2)$, $(3,1,1)$, $(2,2,1)$, $(2,1,1,1)$, and $(1,1,1,1,1)$,
so that $p(5)=7$. Euler's generating function for $p(n)$ is given by
\begin{align*}
P(q) &:= \sum_{n\ge0} p(n)q^n
=
\frac{1}{(q)_\infty},
\end{align*}
where $(a)_m = (a;q)_m := \prod_{j=0}^{m-1} (1-aq^j)$
for $m\in \N_0\cup\{\infty\}$.

Determining the growth of $p(n)$ was one of the early motivating problems in the theory of partitions, and Hardy and Ramanujan \cite{HR} developed the Circle Method in order to
provide an asymptotic series expansion for $p(n)$.
Their result includes the asymptotic main term
\begin{equation*}
p(n)\sim \frac{1}{4\sqrt{3}n}e^{\pi\sqrt{\frac{2n}{3}}}\quad
\textnormal{as }\quad  n\rightarrow\infty.
\end{equation*}
The proof intrinsically relies on the fact that $P(q)$ is essentially the inverse of Dedekind's $\eta$-function,
which is defined by
$\eta(\tau):= q^{\frac{1}{24}}(q)_\infty$ (throughout $q:=e^{2\pi i\tau}$), which is a weight $\frac12$ modular form (with multiplier).

Ramanujan \cite{R} also used modular forms to prove striking congruences for the partition function. He showed that if $n \in\N_0$, then
\begin{align}\label{E:Ram}
p(5n+4) \equiv 0 \pmod{5}, \quad
p(7n+5) \equiv 0 \pmod{7}, \quad
p(11n+6) \equiv 0 \pmod{11}.
\end{align}
Dyson \cite{Dys} introduced a more combinatorial approach to these congruences, where the partitions are decomposed into equal classes based on certain statistics. He defined the {\it rank} of a partition to be its largest part minus the number of its parts, so that
\begin{equation*}
\rank(\lambda) := \lambda_1 - \ell(\lambda).
\end{equation*}
Andrews and Garvan \cite{AG} later defined the {\it crank} of a partition as follows. Let $o(\lambda)$ denote the number of ones in $\lambda$ and
define $\mu(\lambda)$ to be the number of parts strictly larger than $o(\lambda)$. Then if $|\lambda| > 1$,
\[
\crank(\lambda):=\begin{cases}
\text{largest part of }\lambda\quad&\text{if }o(\lambda)=0,\\
\mu(\lambda)-o(\lambda)&\text{if }o(\lambda)>0.
\end{cases}
\]
For the empty partition, it is convenient to define $\rank(\emptyset) = \crank(\emptyset) := 1$. The crank of the partition $\lambda = 1$ does not follow the above definition, as described below in \eqref{E:Cgen}. As the focus of the present paper is not on divisibility properties, the Ramanujan congruences
are not discussed in the sequel; we simply mention here that the rank can be used to decompose the partitions counted by \eqref{E:Ram} into $k$ equal classes for $k\in \{5,7\}$ (this was conjectured in \cite{Dys}, and proven by Atkin and Swinnerton-Dyer \cite{ASD}), and the crank decomposes all three congruences from \eqref{E:Ram}. Dyson's conjecture for the crank famously required both defining the statistic and proving the partition decomposition; the first result is due to Garvan \cite{Ga}, who proved a decomposition using weighted ``vector partitions'', and subsequently Andrews and Garvan \cite{AG} defined the partition crank given above.

Let $N(m,n)$ denote the number of partitions of $n$ with rank $m$,
and similarly, for $n\in\Z\setminus\{1\}$, let $M(m,n)$ denote the number of partitions of $n$ with crank $m$ (for $n = 1$, the series below requires defining $M(\pm 1, 1) = 1$ and $M(0,1) = -1$). The generating functions are given as
\begin{align}
\label{E:Rgen}
R(w;q)
&:=
	\sum_{\substack{n\ge0 \\m\in\Z}}N(m,n)w^mq^n
	=
	\sum_{n\ge0}\frac{q^{n^2}}{(wq)_n (w^{-1}q)_n}
= \frac{1-w}{(q)_\infty} \sum_{n \in \Z} \frac{(-1)^{n}
q^{\frac{n(3n+1)}{2}}}{1-wq^n}
,\\
\label{E:Cgen}
C(w;q)
&:=
	\sum_{\substack{n\ge0 \\ m\in\Z}}M(m,n)w^mq^n
	=
	\frac{(q)_\infty}{(wq)_\infty (w^{-1}q)_\infty}
= \frac{1-w}{(q)_\infty} \sum_{n \in \Z} \frac{(-1)^{n}
q^{\frac{n(n+1)}{2}}}{1-wq^n}
.
\end{align}
The first series in \eqref{E:Rgen} is due to Garvan \cite[equation (7.2)]{Ga}, and the final expression was proven by Atkin and Swinnerton-Dyer \cite[Lemma 1]{ASD}.
The product in \eqref{E:Cgen} is in Andrews and Garvan's proof of \cite[Theorem 1]{AG}, and the series was proven by Garvan \cite[equation (7.15)]{Ga} (an equivalent identity also appeared in Ramanujan's Lost Notebook \cite[entry 3.2.1]{AB}).

As mentioned above, the modularity of $P(q)$ is extremely helpful in determining the asymptotic behavior of $p(n)$, and the situation is similar for the crank, as $C(w;q)$ is a modular form (up to a rational power of $q$) when $w$ is a root of unity \cite{Mah05}.
The first author and Ono showed \cite[Theorem 1.1]{BO10} that if $w$ is a root of unity, then $R(w;q)$ is a mock modular form (in the modern sense defined in \cite{Zag09}). As a consequence of these modularity properties, three of the present authors proved asymptotic formulas for the moments of these statistics. For even integers $2k$, the moments of the partition rank and crank are defined by
\begin{equation*}
N_{2k}(n) := \sum_{m \in \Z} m^{2k} N(m,n), \qquad
M_{2k}(n) := \sum_{m \in \Z} m^{2k} M(m,n).
\end{equation*}
Corollary 1.4 of \cite{BMR11} states that as $n \to \infty$,
\begin{equation}
\label{E:NMAsymp}
N_{2k}(n) \sim M_{2k}(n) \sim \frac{e^{\pi \sqrt{\frac{2n}{3}}}}{4\sqrt{3}n} \left(2^{2k}-2\right)\left|B_{2k}\right| (6n)^{k},
\end{equation}
where $B_{n}$ is the $n$-th Bernoulli number. The asymptotic formulas in \eqref{E:NMAsymp} were motivated by Garvan's conjecture \cite{Ga10} that $M_{2k}(n) > N_{2k}(n)$. An asymptotic version of these inequalities is implied by the main result in \cite{BMR11} (which refines \eqref{E:NMAsymp} to provide asymptotic series for the moments), and Garvan gave explicit, combinatorial generating functions for the full inequalities in \cite{Ga11}.

The asymptotic formulas for the moments are also of interest due to their probabilistic implications. Diaconis, Janson, and the fourth author \cite{DJR13} used the Method of Moments in order to show that these moments determine the limiting distribution (as $n \to \infty$) of the rank (and crank) function for a partition of $n$ chosen uniformly at random. Specifically, the main (unnumbered) proposition in \cite{DJR13} states that
\begin{equation}
\label{E:rankDist}
\lim_{n \to \infty} \frac{1}{p(n)} \left| \left\{ |\lambda| = n \; : \; \frac{\rank(\lambda)}{\sqrt{6n}} \leq x \right\} \right| \to F_r(x),
\end{equation}
where $F_r(x) := (1+e^{-\pi x})^{-1}$ is the difference between two independent extreme value distributions.  Moreover, the authors of \cite{DJR13} explained how \eqref{E:rankDist} is consistent with the heuristic formula obtained by assuming that $\lambda_1$ and $\ell(\lambda)$ have independent distributions (which must be identical due to the conjugation map for partitions). In particular, Erd\"os and Lehner's distributional limit for the largest part \cite[Theorem 1.1]{EL41} states that
\begin{equation}
\label{E:pkDist}
\lim_{n \to \infty} \frac{1}{p(n)} \left| \left\{ |\lambda| = n \; : \; \frac{\lambda_1 - A \sqrt{n} \log(A \sqrt{n})}{A\sqrt{n}} \leq x \right\} \right| \to e^{-e^{-x}},
\end{equation}
where $A := \frac{\sqrt{6}}{\pi}$, and $e^{-e^{-x}}$ is a standardized extreme value distribution (cf. \cite[page 195]{billingsley}).

\subsection{Unimodal sequences: History and combinatorial results}

We now consider unimodal sequences of integers, which have many similarities to partitions, and appear widely in enumerative combinatorics (see Stanley's survey article \cite{Stan89} for many examples and applications).
In particular, a sequence of positive integers $\{ a_j\}_{j=1}^s$
is a {\it strongly unimodal sequence of size $n$} (we use braces to distinguish from partitions) if it satisfies
\begin{equation}
\label{E:uniDef}
a_1<\dots< a_{k-1} < a_k> a_{k+1}  >\dots > a_s
\end{equation}
for some $k\in\mathbb N$
and $a_1+\dots +a_s=n$.  If $\sigma$ is a
strongly unimodal sequence, then we denote its size by $|\sigma|$, and for a given $n$, let $\cU(n)$ be the set of all strongly unimodal sequences such that $|\sigma| = n$.  We also denote the enumeration function for strongly unimodal sequences by
 $u(n) := \abs{\cU(n)}$.
As an example, $u(5)=6$, since the strongly unimodal sequences of size $5$ are
$\{5\}$, $\{1,4\}$, $\{4,1\}$, $\{2,3\}$, $\{3,2\}$, and $\{1,3,1\}$.
The generating function for $u(n)$ is given by \cite[page 68]{An11}
\begin{equation*}
U(q) := \sum_{n\geq 1} u(n) q^n = \sum_{n \geq 0 }  (-q)_n^2q^{n+1}.
\end{equation*}

To our knowledge, strongly unimodal sequences were first introduced by Andrews in \cite{An11}, where he used the terminology ``strictly convex compositions'' (with enumeration function $X_d(n)$). Andrews' main result \cite[equation (1.5)]{An11} expresses $U(q)$ in terms of certain mock theta functions. The study of other related sequences has a more extensive history. For example, a {\it unimodal sequence} of integers satisfies a modified version of \eqref{E:uniDef} where the inequalities must no longer be strict; various combinatorial and number-theoretic properties of unimodal sequences were proven in \cite{Aul51,Wri68}. See \cite{BM14JCTA} for the history of several other variants of unimodal sequences.

The fourth author \cite{rhoades} exploited a connection between $U(q)$
and mixed mock modular forms (linear combinations of modular forms multiplied by mock modular forms) using a
technique developed by the first and the third authors \cite{BM} in order to give an asymptotic series for $u(n)$ (note that the enumeration function for strongly unimodal sequences is denoted by $u^*(n)$ in \cite{rhoades}).  As a consequence one has for any $N\in\mathbb N_0$,
\begin{equation*}
u(n) = \frac{1}{8 \cdot 6^\frac14 n^\frac34} e^{\pi\sqrt{\frac{2n}{3}}} \left( 1+ \sum_{1\leq r\leq N}
\frac{\beta_r}{n^{\frac{r}{2}}} + O\left(n^{-\frac{N+1}{2}}\right)\right)
\end{equation*}
for explicitly computable $\beta_r$ (e.g.  $\beta_1 = -\frac{2\pi^2+9}{2^6  \sqrt{24} \pi}$).

The {\it rank} of a strongly unimodal sequence is the number of terms after the maximal term minus the number
of terms that precede it, i.e., in the notation above, the rank is $s-2k+1$.
By letting $w$ (resp. $w^{-1}$)
keep track of the terms after (resp. before)
a maximal term,  we have that $u(m,n)$, the number of
size $n$ and rank $m$ strongly unimodal sequences, satisfies (see \cite[equation (1.1)]{BOPR})
\begin{equation*}
U(w;q):=\sum_{\substack{n\geq 0 \\ m\in\mathbb Z}}u(m,n)w^mq^n=
\sum_{n\geq 0}(-wq)_n \left(-w^{-1}q\right)_n q^{n+1}.
\end{equation*}
\begin{remark}
The unimodal rank directly coincides with the partition rank for a different class of unimodal sequences. Section 2.2 of \cite{BM14JCTA} gives a bijection between partitions and ``receding stacks with summits'', which are certain unimodal sequences where repeated parts are allowed. In particular, the important feature to this discussion is that a partition $\lambda$ is mapped to a unimodal sequence $\sigma = \{a_1, \dots, a_{\ell(\lambda)}, b_{\lambda_1 - 1}, \dots, b_1\}$, where $a_{\ell(\lambda)}$ is the largest part. Thus
\begin{equation*}
\rank(\sigma) = (\lambda_1 - 1) - (\ell(\lambda) - 1) = \rank(\lambda).
\end{equation*}
\end{remark}

By definition, specializing the refined enumeration function to $w=1$ yields
$$ U(1;q)=U(q).$$
However, as with the partition rank generating function $R(w;q)$, the unimodal rank generating function also has interesting analytic properties at other roots of unity. For example, Bryson, Ono, Pitman, and the fourth author \cite{BOPR} showed that setting $w = \pm i$ gives the third order mock theta function
$$ U(\pm i;q) = \Psi(q):= \sum_{n \geq 1} \frac{q^{n^2}}{\left(q;q^2\right)_n}.$$

In this paper we are interested in combinatorial identities, asymptotic enumeration formulas, and the limiting probabilistic distribution of the rank statistic for strongly unimodal sequences. By symmetry, it is clear that
\begin{equation}\label{symmetry}
u(m,n) = u(-m,n).
\end{equation}
For this reason, we only consider $m\ge0$ throughout the article.
The following table gives the first few values of $u(m,n)$:
\begin{table}[h]
\label{T:umn}
\caption{Values of $u(m,n)$}
\begin{center}
\begin{tabular}{c || rrrrrrrrrrrrrrrrrrrr}
\diagbox[height=0.7cm, width = 0.7cm]{$\!\!\!m$}{$n\!\!$} & 1 & 2 & 3 & 4 & 5 & 6 & 7 & 8 & 9 & 10 & 11 & 12 & 13 & 14 & 15 & 16 & 17 & 18 & 19 & 20\\
\hline \hline
0 & 1& 1& 1 & 2 & 2 &4 & 5 & 7 & 10 & 13 & 17 & 24 & 31 & 40 & 53 & 69 & 88 & 113 & 144 & 183  \\
1 &&&1 & 1 & 2 & 2 & 4 & 5 & 7 & 10 & 14 & 18 & 25 & 33 & 43 & 56 & 73 & 94 & 121 & 153 \\
2&&&&&& 1 & 1 & 2 & 3 & 4 & 6 & 9 & 12 & 16 & 23 & 30 & 40 & 53 & 69 & 90\\
3 &&&&&&&&&& 1 & 1 & 2 & 3 & 5 & 6 & 10 & 13 & 19 & 25 & 34  \\
4 &&&&&&&&&&&&&&& 1 & 1 & 2 & 3 & 5 & 7\\

\end{tabular}
\end{center}
\end{table}


As an initial observation, we see that the first few non-zero values of $u(m,n)$ for fixed $m$
are equal to the values of the partition function. In particular, the table suggests that
$u(m, \frac12(m+2)(m+1) + n) = p(n)$ for $ 0 \le n \le m+1$. The following theorem explains this phenomenon.

\begin{theorem}\label{thm:fixedm}
We have the following generating function for $m\in\N_0$
\begin{equation}\label{defineUm}
U_m(q):=\sum_{n\geq  1} u(m, n) q^n = \frac{q^{\frac{m(m+1)}{2}}}{(q)_\infty}
\sum_{n\ge1} \frac{(-1)^n q^{\frac{n(n+1)}{2}+mn}}{1-q^{n+m}} \left(q^{n(n+m)}  -1\right).
\end{equation}
In particular, $u(m, \frac12(m+2)(m+1) + n) = p(n)$ for $ 0 \le n \le m+1$.
\end{theorem}

Comparing
\eqref{defineUm}
to \eqref{E:Rgen} and \eqref{E:Cgen}, it is not surprising that unimodal sequences are closely related to the generating functions for partition ranks and cranks, although this is certainly not clear from the combinatorial definitions.
The precise relationship is described in the following corollary to Theorem \ref{thm:fixedm}.

\begin{corollary}\label{cor:ospt}
We have
$$\sum_{n\ge 1} u(0, n) q^n = \sum_{n\ge 1} \ospt(n) q^n,$$
where $$\ospt(n) := \sum_{\substack{   |\lambda| = n \\ \crank(\lambda) >0 }} \crank(\lambda) -  \sum_{\substack{ |\lambda| = n \\ \rk(\lambda) >0}} \rk(\lambda).$$
\end{corollary}

The $\ospt$-function was introduced by Andrews, Chan, and Kim in \cite{ACK}, where they provided a more combinatorial proof of Garvan's inequalities for even crank and rank moments, and also introduced a natural variant for (positive) odd moments. In particular, the $\ospt(n)$ function above is essentially the difference of the {\it first} moments of the partition crank and rank statistics. They also gave a combinatorial interpretation
of $\ospt(n)$ in terms of so-called even and odd strings in the partitions of $n$ \cite[Theorem 4]{ACK}.
From their interpretation, it is clear that $\ospt(n)\ge0$.

Furthermore, the asymptotic behavior of the $\ospt$-function was determined by the first and third author, who proved in \cite[Theorem 1.4]{BM12TAMS} that $\ospt(n)\sim\frac{p(n)}{4}$. In turn, Chan and Mao studied the combinatorial relationship between the $\ospt$-function and partitions; one of their main results \cite[equation (1.9)]{CM} proves that $\ospt(n)<\frac{p(n)}{2}$ for $n\ge3$. Corollary \ref{cor:ospt} provides an alternative combinatorial interpretation of the $\ospt$-function. By
examining rank zero strongly unimodal sequences, we obtain the following refinement of Chan and Mao's inequality.

\begin{theorem}\label{thm:ospt}
For $n\ge2$ we have
$$\ospt(n)\le \frac{p(n)-M(0,n)}{2}.$$
We note that
$M(0,n)$ is positive for $n\ge3$.
\end{theorem}

One of the striking features of the columns in Table \ref{T:umn} is that $m\mapsto u(m,n)$ appears to be unimodal.
In fact, additional numerical data (checked by MAPLE for all $n \leq 500$) suggests that a stronger property holds. Recall that a sequence of positive real numbers $\{a_m\}_{m = M}^N$ ($M,N \in \Z$) is {\it log-concave} if $a_m^2 - a_{m-1} a_{m+1} \geq 0$ for all $M+1 \leq m \leq N-1$. It is a straightforward fact that if $\{a_m\}$ is a {\it symmetric} sequence ($a_{-m} = a_m$ for all $0 \leq m \leq M$) and log-concave, then it is unimodal with peak $a_0$. See \cite{Stan89} for further discussion of log-concave sequences. We offer the following conjecture.

\begin{conjecture}\label{pgeospt}
For $n\geq \max(7,\frac{|m|(|m|+1)}{2}+1)$ we have
$$u(m, n)^2 > u(m-1, n) u(m+1, n).$$
\end{conjecture}
\begin{remarks}
{\it 1.} Conjecture \ref{pgeospt} states that $\{u(m,n)\}_m$ is strictly log-concave for $n > 6$ (and hence strictly unimodal). The data in Table \ref{T:umn} shows that for $n \leq 6$ the sequence is log-concave, but not necessarily strict; for example, $u(1,6)^2 - u(0,6) u(2,6) = 0$.

{\it 2.} Since by Corollary \ref{cor:ospt} we have that $u(0,n) = \ospt(n)$,
it is natural to ask if there are other combinatorial interpretations of $u(m,n)$ for fixed $m\geq 1$.
Such interpretations may give insight into Conjecture \ref{pgeospt}.

{\it 3.} There have been a number of recent results on the log-concavity of partition enumeration functions; for example, DeSalvo and Pak \cite[Theorem 1.1]{DP15} proved that $\{p(n)\}_{n \geq 25}$ is log-concave.
\end{remarks}

As further evidence for the unimodality/log-concavity of the values of $u(m,n)$, we prove the following asymptotic version.
\begin{theorem}\label{thm:unimodalAsymp}
For fixed $m\in \N_0,$ we have as $n\rightarrow\infty$
\begin{align*}
u(m,n) &\sim \frac{1}{16\sqrt{3}n} e^{\pi\sqrt{\frac{2n}{3}}},\\
u(m, n) - u(m+1, n)&\sim  \frac{\pi(2m+1)}{96 \sqrt{2} n^{\frac{3}{2}} }
e^{\pi \sqrt{\frac{2n}{3}}},
\\
u(m,n)^2 - u(m-1,n)u(m+1,n)
&\sim
	\frac{\pi}{768\sqrt{6}n^{\frac52}}e^{2\pi\sqrt{\frac{2n}{3}}}
.
\end{align*}
In particular, Conjecture \ref{pgeospt} is true for sufficiently large $n$.
\end{theorem}

\begin{remark}
We prove Theorem \ref{thm:unimodalAsymp} using Wright's Circle Method \cite{Wri33},
which naturally gives asymptotic expansions of the form, for $N\in\mathbb N_0$,
\begin{equation*}
u(m,n) = \frac{e^{\pi \sqrt{\frac{2n}{3}}}}{16\sqrt{3} n} \left(1 + \sum_{1\leq r\leq N} \frac{\alpha_r(m)}{n^{\frac{r}{2}}} +O\left(n^{-\frac{N+1}{2}}\right)\right).
\end{equation*}
Here all of the $\alpha_r(m)$ are explicitly computable, although the asymptotic terms stated in Theorem \ref{thm:unimodalAsymp}
only require the part of $\alpha_1$ that depends on the value of $m$.
\end{remark}

\subsection{Asymptotic results for the rank of strongly unimodal sequences}

We now consider the values of the rank statistic among all $\sigma \in \cU(n)$ for large $n$. We calculate the moments of the rank, and then appeal to the probabilistic ``Method of Moments'' in order to describe the limiting distribution of the rank. For $k \in\N_0$, define
$$
u_{2k}(n) := \sum_{m\in\Z} m^{2k} u(m,n).
$$
Note that \eqref{symmetry} implies that the analogous odd moments satisfy $u_{2k +1}(n) = 0$.
The following theorem provides the asymptotic behavior of the even moments, where we use the double factorial notation
$(2k-1)!!:= 1\cdot 3\cdot \ldots \cdot (2k-3) \cdot (2k-1)$.

\begin{theorem}\label{thm:distributionAsymp}
For each $k\in\mathbb N_0$, we have
\begin{equation*}
u_{2k}(n) \sim \frac{e^{\pi \sqrt{\frac{2n}{3}}}}{8 \cdot 6^{\frac14} n^{\frac34}} (2k -1)!!\left(\frac{6n}{\pi^2}\right)^{\frac{k}{2}}.
\end{equation*}
\end{theorem}

Recall the relationship between \eqref{E:rankDist} and \eqref{E:pkDist}, where the asymptotic distribution of the largest part of a partition suggested the natural shape of the distribution for the partition rank. The shape of Theorem \ref{thm:distributionAsymp} is similarly predicted by the close relationship between partitions into distinct parts and
strongly unimodal sequences, as there is a map from pairs $(\lambda, \mu)$ of such partitions to a strongly unimodal sequence given by
\begin{equation}
\label{E:lambdamu}
(\lambda, \mu) \mapsto \{\lambda_{\ell(\lambda)}, \dots, \lambda_2, \lambda_1, \mu_1, \mu_2, \dots \mu_{\ell(\mu)}\}.
\end{equation}
This map is only defined if the largest parts of $\lambda$ and $\mu$ are different, and is then in fact two-to-one onto
the set of strongly unimodal sequences. Denoting the number of such pairs of total size $n$ by $q_2(n)$, with corresponding generating function $\sum_{n\geq 0} q_2(n) q^n = (-q)_\infty^2$,
it is known that $u(n) \sim \frac{q_2(n)}{2}$ (for example, this follows immediately from \eqref{E:1/qinfty} and Theorem \ref{TauberianT} below).

The rank of the unimodal sequence in \eqref{E:lambdamu} is $\ell(\mu) - \ell(\lambda) \pm 1$ (depending on which partition contributes the peak), and thus it is relevant to understand the typical number of parts in a partition into distinct parts. Let $\cQ(n)$ denote the set of partitions into distinct parts of size $n$, and let $q(n) := |\cQ(n)|$ be the enumeration function. Szekeres \cite[Theorem 1]{szekeres} proved that for large $n$, if one picks $\lambda \in \cQ(n)$ uniformly at random, then $\ell(\lambda)$ is normally distributed, with mean $r_0 = \frac{2 \sqrt{3} \log{(2)}}{\pi} \sqrt{n}$ and variance $s^2 = \frac{\sqrt{3}}{\pi} (1 - (\frac{2 \sqrt{3} \log{(2)}}{\pi})^2) \sqrt{n}$.
As a rough estimate, we should therefore expect that for $\sigma \in \cU(n)$, we have $\rk(\sigma) = \ell(\mu) - \ell(\lambda) \pm 1$ for some $\mu \in \cQ(n_1)$ and $\lambda \in \cQ(n_2)$ such that $n_1 \sim n_2 \sim \frac{n}{2}$. This follows from Hardy and Ramanujan's famous asymptotic formula $\log{(q(n))} \sim \pi \sqrt{\frac{n}{3}}$ (see \cite[p. 109]{HR},) which implies that almost all $(\mu, \lambda)$ such that $|\mu| + |\lambda| = n$ satisfy $|\mu| \sim |\lambda| \sim \frac{n}{2}.$

It is a straightforward fact \cite[Example 20.6]{billingsley} that if, for $j \in \{ 1,2\}$, $X_j$ are independent normal random variables with mean $m_j$ and variance $\sigma_j^2$, respectively, then $X_1 - X_2$ is a normal random variable with mean $m_1 - m_2$ and variance $\sigma_1^2 + \sigma_2^2$. In our setting
we know by symmetry that $\rk(\sigma)$ has mean zero, and the heuristic described above suggests that it should be normally distributed with variance approximately $2s^2$, which is of order $\sqrt{n}$ (due to the combinatorial limitations of this rough model, we should not necessarily expect to obtain the precise constant).

Indeed, this prediction is confirmed by Theorem \ref{thm:distributionAsymp}, as we see that it may be equivalently written as
$$
\frac{u_{2k}(n)}{u(n) \left(\frac{6n}{\pi^2}\right)^\frac{k}{2}}
\sim (2k-1)!!.
$$
This matches the values of the even moments for the standard normal distribution (cf. Example 21.1 in \cite{billingsley}), and we therefore conclude that for large $n$, the rank is normally distributed around zero with variance $\frac{\sqrt{6n}}{\pi}$.
\begin{corollary}
\label{C:normal}
For all $x \in \R$, we have
$$\lim_{n\to \infty} \frac{1}{u(n)} \left| \left\{ \sigma \in \cU(n) : \frac{\textnormal{rank}(\sigma)}{\left(\frac{6n}{\pi^2}\right)^{\frac14}} \le x \right\} \right| = \Phi(x)$$
where $\Phi(x) := \frac{1}{\sqrt{2\pi}} \int_{-\infty}^x e^{-\frac{u^2}{2}} du$.
\end{corollary}
\begin{remark}
As a consequence of Corollary \ref{C:normal}, we have that, for $a,b\in\R$ with $a\leq b$,
$$
\lim_{n\to \infty} \frac{1}{u(n)} \left| \left\{ \sigma \in \cU(n) : a \leq \frac{\textnormal{rank}(\sigma)}{\left(\frac{6n}{\pi^2}\right)^{\frac14}} \le b \right\} \right| = \Phi(b) - \Phi(a),
$$
which tends to one as $b \to \infty$ and $a \to -\infty$. This means that for any $\varepsilon > 0$, ``almost all'' strongly unimodal sequences $\sigma$ have $|\rk(\sigma)| < n^{\frac14 + \varepsilon}$ (recall that the maximum value of the rank is roughly $\sqrt{|\sigma|}$).
\end{remark}

We can use Corollary \ref{C:normal} to determine the asymptotic behavior of the absolute moments for the rank. For $r \in \N_0$, define the {\it absolute moments}
$$
u_r^+(n) := \sum_{m \in\Z} |m|^r u(m,n).
$$
Note that the even absolute moments are already described by Theorem \ref{thm:distributionAsymp}, as $u_{2k}^+(n) =  u_{2k}(n)$.

\begin{corollary}
\label{C:posmoments}
As $n \to \infty$,
\begin{equation*}
\frac{u_r^+(n)}{u(n) \left(\frac{6n}{\pi^2}\right)^\frac{r}{4}} \sim \frac{2^{\frac{r}{2}}}{\sqrt{\pi}} \Gamma\left(\frac{r+1}{2}\right).
\end{equation*}
\end{corollary}

\begin{remark} Unlike Theorem \ref{thm:unimodalAsymp}, where we can obtain an asymptotic expansion in $n^{-\frac12}$ with an arbitrary number of terms using Wright's Circle Method, we do not have any control over the error terms in Corollary \ref{C:posmoments} due to the weaker notions of convergence used in the Method of Moments.
\end{remark}

The paper is organized as follows.
In Section \ref{sec:GeneratingFunction} we consider the combinatorial properties of strongly unimodal sequences and related generating functions, and prove Theorem \ref{thm:fixedm}, Corollary \ref{cor:ospt}, and Theorem
\ref{thm:ospt}.
In Section \ref{sec:Unimodal} we establish Theorem \ref{thm:unimodalAsymp}, giving an asymptotic
version of Conjecture \ref{pgeospt}.
In Section \ref{sec:Distribution}
we determine the asymptotic behavior of the moments of the rank statistic, proving Theorem \ref{thm:distributionAsymp}, Corollary \ref{C:normal}, and Corollary \ref{C:posmoments}.
Finally, in Section \ref{sec:Modularity} we discuss the modularity properties of the generating function for strongly unimodal sequences, and the relation to previously studied examples of mock modular and quantum modular forms.

\section*{Acknowledgments}
The authors thank Don Zagier and Sander Zwegers for insightful discussions
concerning quantum modular properties of the functions $U_m$ defined in
\eqref{defineUm}, and we thank Irfan Alam for discussion on the Method of Moments.
Moreover we thank the referee for helpful comments on an earlier version of this paper,
which greatly improved the exposition.

\section{Proof of Theorem \ref{thm:fixedm}, Corollary \ref{cor:ospt}, and Theorem \ref{thm:ospt}}\label{sec:GeneratingFunction}
In this section, we give some basic results for the generating function $U(w;q)$.
We prove Theorem \ref{thm:fixedm},  Corollary \ref{cor:ospt}, and Theorem \ref{thm:ospt}.

\begin{proof}[Proof of Theorem \ref{thm:fixedm}]	
From Entry 3.4.7 of \cite{AB} (or equivalently Theorem 4 of \cite{Choi})
and Lemma 7.9 of \cite{Ga}, one can conclude the following identity,
\begin{align}
\label{E:UAppell}
U(w;q) (q)_\infty
=&- \frac{1}{1+w^{-1}} \( \sum_{n\in\mathbb Z\backslash \{0\}  } \frac{(-1)^n q^{\frac{n(3n+1)}{2}}}{1+wq^n}
 -  \sum_{n\in\mathbb Z\backslash \{0\} } \frac{w^{-n} q^{\frac{n(n+1)}{2}}}{1+wq^n} \).
\end{align}
We let $[w^m]F(w;q)$ denote the coefficient of $w^m$ in $F(w;q)$, where $F(w;q)$ is a
series in $w$ and $q$. Our goal is to determine $U_m(q) = [w^m]U(w;q)$.
In order guarantee the absolute convergence of the various $q$-series that appear throughout
the proof, we henceforth assume that $0 < |q| < |w| < 1$.

We begin by considering the first summation in \eqref{E:UAppell}, and expand both denominators
as geometric series to obtain
\begin{align*}
\frac{1}{1+w^{-1}} & \sum_{n\in\mathbb Z\backslash \{0\}  } \frac{(-1)^n q^{\frac{n(3n+1)}{2}}}{1+wq^n}
=
	\frac{w}{1+w} \left(\sum_{n\geq 1} \frac{(-1)^n q^\frac{n(3n+1)}{2}}{1+wq^n}
	+ w^{-1} \sum_{n\geq 1} \frac{(-1)^n q^\frac{n(3n+1)}{2}}{1+w^{-1}q^n}\right)
\\
&= 	
	\sum_{j,\ell \geq 0, n\geq 1}
		(-1)^{n+j+\ell} w^{j+\ell+1} q^{\frac{n(3n+1)}{2} + nj}
	+
	\sum_{j,\ell\geq 0, n\geq 1}
		(-1)^{n+j+\ell} w^{-j+\ell} q^{\frac{n(3n+1)}{2} + nj}
\\
&=
	\sum_{n\geq 1, j\geq 0, \ell\ge j+1}
		(-1)^{n+\ell+1} w^{\ell} q^{\frac{n(3n+1)}{2} + nj}
	+
	\sum_{n\geq 1, j\geq 0, \ell\ge -j }
		(-1)^{n+\ell} w^{\ell} q^{\frac{n(3n+1)}{2} + nj}
.
\end{align*}
Thus, for $m\ge0$, we have
\begin{align}\label{E:U1stsum}
[w^m]\frac{1}{1+w^{-1}} & \sum_{n\in\mathbb Z\backslash \{0\}  } \frac{(-1)^n q^{\frac{n(3n+1)}{2}}}{1+wq^n}
=
	\sum_{n\ge1}(-1)^{n+m}q^{\frac{n(3n+1)}{2}}\left(
		-\sum_{0\le j\le m-1}q^{nj}
		+\sum_{j\ge0}q^{nj}	
	\right)
\nonumber\\
&=
	\sum_{n\ge1}(-1)^{n+m}q^{\frac{n(3n+1)}{2}}
	\sum_{j\ge m}q^{nj}
=
	\sum_{n\ge1}\frac{(-1)^{n+m}q^{\frac{n(3n+1)}{2}+nm}}{1-q^n}
.
\end{align}

The second sum from \eqref{E:UAppell} is expanded in a similar manner.
Again using the geometric series, we obtain
\begin{align*}
\frac{1}{1+w^{-1}} & \sum_{n\in\mathbb Z\backslash \{0\}  } \frac{w^{-n} q^{\frac{n(n+1)}{2}}}{1+wq^n}
=
	\frac{w}{1+w} \left(\sum_{n=1}^\infty \frac{w^{-n} q^{\frac{n(n+1)}{2}}}{1+wq^n}
		+ w^{-1}\sum_{n=1}^\infty \frac{w^nq^{\frac{n(n+1)}{2}}}{1+w^{-1}q^n}\right)
\\
&=
	\sum_{j,\ell\ge 0, n\ge1}(-1)^{j+\ell}w^{-n+\ell+j+1}q^{\frac{n(n+1)}{2}+nj}
	+
	\sum_{j,\ell\ge 0, n\ge1}(-1)^{j+\ell}w^{n+\ell-j}q^{\frac{n(n+1)}{2}+nj}
\\
&=
	\sum_{n\ge1, j\ge 0, \ell\ge-n+j+1 }(-1)^{n+\ell+1}w^{\ell}q^{\frac{n(n+1)}{2}+nj}
	+
	\sum_{n\ge1, j\ge 0, \ell\ge n-j }(-1)^{n+\ell}w^{\ell}q^{\frac{n(n+1)}{2}+nj}
.
\end{align*}
Thus, for $m\ge0$, we have
\begin{align}\label{E:U2ndsum}
[w^m]\frac{1}{1+w^{-1}} & \sum_{n\in\mathbb Z\backslash \{0\}  } \frac{w^{-n} q^{\frac{n(n+1)}{2}}}{1+wq^n}
=
	\sum_{n\ge1} (-1)^{n+m}q^{\frac{n(n+1)}{2}}\left(
		-\sum_{0\le j\le m+n-1}q^{nj}
		+\sum_{j\ge \max(0,n-m)}q^{nj}
	\right)
\nonumber\\
&=
	-\sum_{n\ge1} \frac{(-1)^{n+m}q^{\frac{n(n+1)}{2}}\left(1-q^{n(n+m)}\right) }{1-q^n}
	+
	\sum_{1\le n\le m} \frac{(-1)^{n+m}q^{\frac{n(n+1)}{2}}}{1-q^n}
	\nonumber\\&\quad
	+
	\sum_{n\ge m+1} \frac{(-1)^{n+m}q^{\frac{n(n+1)}{2}+n(n-m)} }{1-q^n}
\nonumber\\
&=
	\sum_{n\ge1} \frac{(-1)^{n+m}q^{\frac{n(3n+1)}{2}+nm}}{1-q^n}
	+
	\sum_{n\ge m+1} \frac{(-1)^{n+m}q^{\frac{n(n+1)}{2}} \left(q^{n(n-m)}-1\right) }{1-q^n}
.
\end{align}

By equations \eqref{E:UAppell}, \eqref{E:U1stsum}, and \eqref{E:U2ndsum} we find that
for $m\ge0$,
\begin{align}
\label{E:Umfinal}
\sum_{n\ge1} u(m,n)q^n
&=
	[w^m]U(w;q)
=
	\frac{1}{(q)_\infty}
	\sum_{n\ge m+1} \frac{(-1)^{n+m}q^{\frac{n(n+1)}{2}} \left(q^{n(n-m)}-1\right) }{1-q^n}
\notag
\\
&=
	\frac{q^{\frac{m(m+1)}{2}}}{(q)_\infty}
	\sum_{n\ge 1} \frac{(-1)^{n}q^{\frac{n(n+1)}{2}+nm} \left(q^{n(n+m)}-1\right) }{1-q^{n+m}}
,
\end{align}
which is the claimed expression for $U_m(q)$.

Now observe that the $n=1$ term in the sum from \eqref{E:Umfinal} reduces to $q^{m+1}$, and thus
\begin{align*}
\sum_{n\ge 1} \frac{(-1)^{n}q^{\frac{n(n+1)}{2}+nm} \left(q^{n(n+m)}-1\right) }{1-q^{n+m}}
&=
q^{m+1} +O\left(q^{2m+3}\right).
\end{align*}
This implies that
\begin{align*}
\sum_{n\ge1} u(m,n)q^n
&=
\frac{q^{\frac{(m+1)(m+2)}{2}}}{(q)_\infty}
\left(1+O\left(q^{m+2}\right)\right),
\end{align*}
and thus
$u(m, \frac12(m+2)(m+1) + n) = p(n)$ for $ 0 \le n \le m+1$.
\end{proof}

We immediately obtain the relation between unimodal sequences and the ospt-function.
\begin{proof}[Proof of Corollary \ref{cor:ospt}]
The proof follows directly from Theorem \ref{thm:fixedm} using the identity
\begin{equation*}
\sum_{n\geq0} \ospt(n) q^n = \frac{1}{(q)_\infty} \sum_{n \geq 1} \left(\frac{(-1)^{n+1} q^{\frac{n(n+1)}{2}}}{1-q^n} -
\frac{(-1)^{n+1} q^{\frac{n(3n+1)}{2}}}{1-q^n}\right),
\end{equation*}
which is Theorem 1 of \cite{ACK}.
\end{proof}

We conclude this section with the proof of Theorem \ref{thm:ospt}.

\begin{proof}[Proof of Theorem \ref{thm:ospt}]
Define a subset of pairs of partitions into distinct parts by $\sS := \{ (\mu, \nu):  \ell(\mu) = \ell(\nu) + 1\} \cup (\emptyset, \emptyset).$ There is a simple injection that maps a strongly unimodal sequence with rank zero to $\sS$. In particular, suppose that $\sigma = \{a_1, \dots, a_k, \dots, a_{2k-1}\}$ has peak $a_k$, and define $(\mu, \nu) \in \sS$ by
$$
\mu := (a_{k}, a_{k-1}, \dots, a_1), \quad \nu := (a_{k+1}, \dots, a_{2k-1}).
$$
This is an invertible injection, as its image consists of all $(\mu, \nu) \in \sS$
such that the largest part in $\mu$ is larger than all parts in $\nu$. Consider the generating function
$$
S(q) := \sum_{n \geq 0} s(n) q^n = \sum_{(\mu, \nu) \in \sS} q^{|\mu| + |\nu|} = 1 + \sum_{n \geq 1} \frac{q^{\frac{n(n+1)}{2}}}{(q)_{n}} \frac{q^{\frac{n(n-1)}{2}}}{(q)_{n-1}},
$$
and recall the following representation for the generating functions for partitions (equation (2.2.9) in \cite{And98})
and partitions with crank zero (Theorem 5 of \cite{Ka}):
\begin{align*}
\sum_{n \geq 1} p(n) q^n = \sum_{n \geq 1} \frac{q^{n^2}}{(q)_n^2},\qquad \quad \sum_{n \geq 1} M(0,n) q^n = (1-q)\sum_{n \geq 1} \frac{q^{n^2+2n}}{(q)_n^2}.
\end{align*}
Then
\begin{align*}
\sum_{n\ge1} p(n)q^n-2\sum_{n\ge1}s(n)q^n
&=
	\sum_{n\ge1}\frac{q^{n^2}}{(q)_n^2}(1-2(1-q^n))=
	-\sum_{n\ge1}\frac{q^{n^2}}{(q)_{n-1}^2}
	+\sum_{n\ge1}\frac{q^{n^2+2n}}{(q)_{n}^2}
\\
&=
	-\sum_{n\ge0}\frac{q^{n^2+2n+1}}{(q)_{n}^2}
	+\sum_{n\ge1}\frac{q^{n^2+2n}}{(q)_{n}^2}
=
	-q+(1-q)\sum_{n\ge1}\frac{q^{n^2+2n}}{(q)_{n}^2}\\
	&=
	-q+\sum_{n\ge1}M(0,n)q^n
.
\end{align*}
Thus, for $n\ge2$,
\begin{align}\label{E:psIden}
s(n)&=\frac{p(n)-M(0,n)}{2}.
\end{align}
In particular, for $n\ge2$ we have the inequality
\begin{align*}
u(0,n) \leq s(n) = \frac{p(n)-M(0,n)}{2}.
\end{align*}

\end{proof}

\begin{remark}
It is also not difficult to achieve minor improvements of our results by describing the image in $\sS$ more precisely;
for example, by considering partitions in $\sS$ of the form
$\mu=(j,1)$ and $\nu=(n-j-1),$ for $2 \leq j \leq \lfloor\frac{n-1}{2}\rfloor$, we obtain
$u(0,n)\le s(n)-\lfloor\frac{n-1}{2}\rfloor+1$
for $n\ge4$.
However, such special cases do not seem to lead to a qualitative improvement of the bound.
We can also determine the asymptotic relationship between $s(n)$ and $\ospt(n)$.
Using \eqref{E:psIden}, we find that $s(n) \sim \frac12 p(n)$, since it is known that
$M(0,n)\sim \frac{\pi p(n)}{4\sqrt{6n}}$ \cite[Corollary 2.1]{KKS},
and so $\ospt(n) \sim \frac12 s(n)$.
\end{remark}

\section{Proof of Theorem \ref{thm:unimodalAsymp}}
 \label{sec:Unimodal}
Our primary goal in this section is to derive the first several terms in the asymptotic expansion for the coefficients of $U_m(q)$, which we achieve using Wright's variant of the Hardy-Ramanujan Circle Method \cite{Wri33, Wri71}.  The proof begins with the determination of the first terms in the asymptotic expansion of the generating function in Section \ref{sec:Unimodal:series}, and then proceeds by estimating its coefficients using a contour integral in Section \ref{sec:Unimodal:coeff}. As before we only consider non-negative $m$ throughout.

\subsection{Asymptotic expansions of generating functions}
\label{sec:Unimodal:series}
Recall Theorem \ref{thm:fixedm} and define
\begin{equation*}
V_m(q) := (q)_\infty U_m(q).
\end{equation*}

The bulk of this section is devoted to determining the asymptotic behavior of $V_m(q)$. We recall a formula for the asymptotic expansion of a series that is a consequence of the Euler-MacLaurin summation formula (here $\b a \in\R^r$, $w\in\C$ with ${\rm Re}(w)>0$, and $F:\C^r\to\C$ is a $\mathcal C^\infty$-function which, along with all of its derivatives, is of rapid decay)
\begin{align}
	\label{E:EulerMac}
\sum_{\b{n} \in \N_0^r} F((&\b{n} + \b{a})w)
\sim
	(-1)^{r}\sum_{\b{n} \in \N_0^r}
	F^{(n_1,\dotsc,n_r)}(\b{0})
	\prod_{j\in\{1,\ldots,r\}} \frac{B_{n_j+1}(a_j)}{(n_j+1)!} w^{n_j}
	\\ \notag
	&+
	\sum_{\mathscr S\subsetneq\{1,\ldots,r\}} \frac{(-1)^{|\mathscr S|}}{w^{r - |\mathscr{S}|}}
	\sum_{\substack{n_j \in \N_0,\\j\in\mathscr{S}}} 	
	\int_{[0,\infty)^{r-|\mathscr S|}} 	
	\left[\prod_{j\in\mathscr S} \frac{\partial^{n_j}}{\partial x_j^{n_j}} F(\b x)\right]_{\substack{x_j=0,\\j\in\mathscr{S}}} 	
	\prod_{k\not\in\mathscr S} dx_k
	\prod_{j\in\mathscr S} \frac{B_{n_j+1}(a_j)}{(n_j+1)!} w^{n_j} ,
\end{align}
where $B_n(x)$ denotes the $n$-th Bernoulli polynomial and throughout the paper we write vectors in bold letters and their components with subscripts. For clarity, we note that $\mathscr{S}$ in \eqref{E:EulerMac} runs over
all proper subsets of $\{1,\dotsc,r\}$, including the empty set.
In particular, the one-dimensional case reduces to
\begin{align*}
\sum_{n \in \N_0} F((n + a)t)
&\sim
	\frac{1}{t}\int_{0}^\infty F(x)dx
	-\sum_{n\geq 0}
	\frac{B_{n+1}(a)}{(n+1)!}F^{(n)}(0) t^{n}
.
\end{align*}

The following proposition gives the first few terms in the asymptotic expansion of $V_m$.

\begin{proposition}
\label{P:Vm}
Suppose that $m \in\mathbb N_0.$  Then as $\tau \to 0$
\begin{equation*}
V_m(q) =\frac14+\left(\frac{m^2}{8}-\frac18\right) 2\pi i\tau + O\left(|\tau|^2\right).
\end{equation*}
\end{proposition}
\begin{proof}
Using finite geometric series, we deduce that
\begin{align}\label{VmPartial}
V_m(q)
&=
\sum_{n_1,n_2\geq0} (-1)^{n_1+n_2}
q^{ \frac{1}{2}\left(n_1+m+\frac{1}{2}\right)^2 + \frac{3}{2}\left(n_2+\frac{1}{2}\right)^2
	+ 2\left(n_1+m+\frac{1}{2}\right)\left(n_2+\frac{1}{2}\right) }
.
\end{align}
We then write
\begin{align}\notag
V_m\left(e^{2\pi i\tau}\right)
&=\sum_{\varepsilon_1,\varepsilon_2\in\{0,1\}} (-1)^{\varepsilon_1+\varepsilon_2} \sum_{n_1,n_2\geq 0} f\left(\sqrt{-2\pi i\tau}\left(n_1+\frac{m}{2}+\frac14+\frac{\varepsilon_1}{2},n_2+\frac14+\frac{\varepsilon_2}{2}\right)\right),\label{VmEulerMcLaurin}
\end{align}
where $f(\b x):=e^{-2x_1^2-6x_2^2-8x_1x_2}$.
We now apply \eqref{E:EulerMac} to the sum on $n_1, n_2$. All terms except those corresponding to the first vanish due to the $(-1)^{\varepsilon_1 + \varepsilon_2}$-factor, and we are therefore left with
\begin{align*}
&V_m\left(e^{2\pi i\tau}\right)=\sum_{\varepsilon_1,\varepsilon_2\in\{0,1\}} (-1)^{\varepsilon_1 + \varepsilon_2} \sum_{n_1,n_2\geq 0} B_{n_1+1}\left(\frac{m}{2}+\frac14+\frac{\varepsilon_1}{2}\right) B_{n_2+1}\left(\frac14+\frac{\varepsilon_2}{2}\right)\\
&\hspace{9.5cm}\times\frac{f^{(n_1,n_2)}(\b 0)}{(n_1+1)!(n_2+1)!}(-2\pi i\tau)^{\frac{n_1+n_2}{2}}.
\end{align*}
Using the facts that $B_n(x)=(-1)^nB_n(1-x)$ and $f^{(n_1,n_2)}(\b 0)=0$ unless $n_1\equiv n_2 \pmod{2}$, we obtain
\[
V_m\left(e^{2\pi i\tau}\right)=2\sum_{n_1,n_2\geq0}\frac{B_{2n_2+1}\left(\frac14\right)\left(B_{2n_1+1}\left(\frac{m}{2}+\frac14\right)-B_{2n_1+1}\left(\frac{m}{2}+\frac34\right)\right)}{(2n_2+1)!(2n_1+1)!}f^{(2n_1,2n_2)}(\b 0)(-2\pi i\tau)^{n_1+n_2}.
\]
Computing the first few terms yields the claim.
\end{proof}
Proposition \ref{P:Vm} enables us to determine the asymptotic behavior of $U_m$ near $q=1$.
\begin{proposition}
\label{P:UmAsymp}
Assume that $\tau=u+iv$, $v=\frac{1}{2\sqrt{6n}}$ and $|u|\leq v$. As $n\to\infty$ we have
\begin{align*}
U_m(q)=\sqrt{-i\tau}e^{\frac{\pi i}{12\tau}}\left(\frac14+\pi i\left(\frac{m^2}{4}-\frac{11}{48}\right)\tau\right)+O\left(n^{-\frac54}e^{\pi\sqrt{\frac{n}{6}}}\right).
\end{align*}
\end{proposition}

\begin{proof}
From the well-known transformation law of the $\eta$-function (e.g., Theorem 3.1 in \cite{Apo90}), one directly concludes the asymptotic formula
\begin{equation}
\label{E:1/qinfty}
\frac{1}{(q)_\infty}=\sqrt{-i\tau} e^{\frac{\pi i}{12}\left(\tau+\frac1{\tau}\right)}\left(1+O\left(e^{-2\pi\sqrt{6n}}\right)\right).
\end{equation}
Expanding $e^{\frac{\pi i\tau}{12}}$ gives the claim.
\end{proof}

We next bound $U_m$ away from the dominant pole $q=1$.
Here, as usual, for sequences $f_n$ and $g_n$ the notation $f_n\ll g_n$ means
that $|f_n|\le c|g_n|$, for sufficiently large $n$ and $c$ a constant.
\begin{proposition}
\label{nondbound}
If $v=\frac1{2\sqrt{6n}}$ and $v\leq |u|\leq\frac12$, then for some $\delta <1 $
\[
\left\lvert  U_m(q)\right\rvert \ll e^{\pi\delta \sqrt{\frac{n}{6}}}.
\]
\end{proposition}

\begin{proof}
We estimate
\[
\left\lvert U_m(q)\right\rvert \ll \frac{1}{|(q)_\infty|}\sum_{n\geq 1} n|q|^{\frac{n^2}{2}}.
\]
The sum on $n$ can be bounded against $$\sum_{n \geq 1} n |q|^n = \frac{|q|}{(1-|q|)^2} \ll |\tau|^{-2}.$$

To estimate $\frac{1}{(q)_\infty}$, we follow Wright's argument from Lemma XVI of \cite{Wri33}. For the convenience of the reader we give the details. First, note that since $|q| < 1$, we have the logarithmic series expansion
\begin{equation*}
\Log{\left(\frac{1}{(q)_\infty}\right)} = \sum_{m \geq 1} \frac{q^m}{m\left(1 - q^m\right)}.
\end{equation*}
The magnitude of this expression is bounded by
\begin{align}
\label{E:log(q)}
\left| \Log\left(\frac{1}{(q)_\infty}\right)\right| \leq \sum_{m \geq 1} \frac{|q|^m}{m \left|1 - q^m\right|}
\leq \sum_{m \geq 1} \frac{|q|^m}{m\left(1 - |q|^m\right)}
- \left(\frac{|q|}{1 - |q|} - \frac{|q|}{|1 - q|}\right).
\end{align}
By \eqref{E:1/qinfty} the final sum in \eqref{E:log(q)} has an asymptotic expansion given by
\begin{equation}
\label{E:log|q|}
\sum_{m \geq 1} \frac{|q|^m}{m\left(1 - |q|^m\right)}
= \log\left(\frac{1}{(|q|;|q|)_\infty}\right)
= \frac{\pi}{12 v} + O(\log{(v)}).
\end{equation}
To estimate the remaining terms in \eqref{E:log(q)}, we compute Taylor series to obtain
\begin{align*}
1 - |q|  = 2 \pi v  \left(1 + O(v)\right),\quad
|1 - q|  = 2 \sqrt{2} \pi v  \left(1 + O(v)\right).
\end{align*}
Indeed, the second identity holds since
\begin{align*}
\left|1-e^{2\pi i(u+iv)}\right|^2&=
	1 - 2\cos(2\pi u)e^{-2\pi v} + e^{-4\pi v}
\geq
	1 - 2 \cos(2 \pi v) e^{-2\pi v} + e^{-4\pi v}\\
&= \left|1 - e^{2 \pi i (v + iv)}\right|^2,
\end{align*}
using the fact that $v \leq |u| \leq \frac12.$ The claim now follows by the Taylor expansion
$$
1 - 2 \cos(2 \pi v) e^{-2\pi v} + e^{-4\pi v}
= 8\pi^2 v^2 + O\left(v^3\right).
$$

Plugging into the last two terms of \eqref{E:log(q)} and combining with \eqref{E:log|q|} implies that
\begin{align*}
\log\left|\frac{1}{(q)_\infty}\right| &
\leq \log\left(\frac{1}{\left(|q|; |q|\right)_\infty}\right) - \frac{|q|}{1 - |q|} + \frac{|q|}{\left|1 - e^{2 \pi i (v + iv)}\right|} \\
& = \frac{\pi}{12 v} - \left(\frac{1}{2 \pi v} - \frac{1}{2 \sqrt{2} \pi v}\right) + O(\log{(v)})= \frac{\pi}{12 v}\left(\!1 - \!\frac{6}{\pi^2}\left(1 - \frac{1}{\sqrt{2}}\right)\!\right) + O(\log{(v)}).
\end{align*}
Thus the claim holds for any $1 - \frac{6}{\pi^2}(1 - \frac{1}{\sqrt{2}}) = 0.8219\ldots < \delta < 1.$
\end{proof}

\begin{remark}
The proof of Proposition \ref{nondbound} also corrects the proof of Corollary 3.4 in the published version of \cite{BM12TAMS}.
\end{remark}

\subsection{Asymptotic behavior of coefficients}
\label{sec:Unimodal:coeff}

Here we use a variant of the Circle Method due to Wright \cite{Wri71}. By Cauchy's Theorem, we obtain
\[
u(m, n)=\frac{1}{2\pi i} \int_{\mathcal{C}}\frac{ U_m(q)}{q^{n+1}}dq=\int_{-\frac12}^{\frac12}  U_m\left(e^{-\frac{\pi}{\sqrt{6n}}+2\pi iu}\right) e^{\pi\sqrt{\frac{n}{6}}-2\pi inu}du,
\]
where $\mathcal{C}$ denotes the circle with radius $e^{-\frac{\pi}{\sqrt{6n}}}$ surrounding the origin counterclockwise. We then split
\[
u(m, n)=I'(n)+I''(n)
\]
with
\begin{align*}
I'(n)&:=\int_{|u|\leq  \frac{1}{2\sqrt{6n}}}  U_m\left(e^{-\frac{\pi}{\sqrt{6n}}+2\pi iu}\right) e^{\pi\sqrt{\frac{n}{6}}-2\pi inu} du,\\
I''(n)&:=\int_{ \frac{1}{2\sqrt{6n}}\leq |u|\leq \frac{1}{2} } U_m\left(e^{-\frac{\pi}{\sqrt{6n}}+2\pi iu}\right) e^{\pi\sqrt{\frac{n}{6}}-2\pi inu} du.
\end{align*}
It turns out that $I'(n)$ contributes the asymptotic main term, whereas $I''(n)$ is part of the asymptotic error term.
To see this, we rewrite
\[
I'(n)=\frac{1}{2\sqrt{6n}}\int_{-1}^1 U_m\left(e^{\frac{\pi}{\sqrt{6n}}(-1+iu)}\right) e^{\pi\sqrt{\frac{n}{6}}(1-iu)}du.
\]
We next approximate $I'(n)$ by a Bessel function. For this define for $s\in\R$
\[
P_s(n):=\frac{1}{2\pi i}\int_{1-i}^{1+i} w^s e^{\pi\sqrt{\frac{n}{6}}\left(w+\frac1w\right)}dw.
\]
We then may write, using Proposition \ref{P:UmAsymp},
\begin{equation*}
I'(n)=\frac{\pi}{8\cdot 2^\frac14\cdot 3^{\frac34}\cdot n^{\frac34}}P_{\frac12}(n)-\frac{\pi^2\left(\frac{m^2}{4}-\frac{11}{48}\right)}{12\cdot 2^{\frac34}\cdot 3^{\frac14}\cdot n^{\frac54}}P_{\frac32}(n) + O\left(n^{-\frac74}e^{\pi\sqrt{\frac{2n}{3}}}\right).
\end{equation*}
We have (see \cite[Section 5]{Wri71}) the following approximation
\[
P_s(n)-I_{-s-1}\left(\pi\sqrt{\frac{2n}{3}}\right)\ll e^{\frac{3\pi}{2}\sqrt{\frac{n}{6}}},\qquad\textnormal{ as } n\to\infty.
\]

We next turn to bounding $I''(n)$.
Using Proposition \ref{nondbound} gives
\[
\left\lvert I''(n)\right\rvert\ll \int_{\frac{1}{2\sqrt{6n}}\leq |u|\leq \frac{1}{2}} \left|U_m\left(e^{-\frac{\pi}{\sqrt{6n}}+2\pi i u}\right)\right|e^{\pi \sqrt{\frac{n}{6}}} du\ll e^{(1 + \delta) \pi \sqrt{\frac{n}{6}}};
\]
the important feature of this bound is that it is exponentially smaller than the initial terms in the asymptotic expansion.

Thus we find that
\begin{align*}
u(m,n)
&=
	\frac{\pi}{8\cdot 2^\frac14\cdot 3^{\frac34}\cdot n^{\frac34}}I_{-\frac32}(n)
	-
	\frac{\pi^2\left(\frac{m^2}{4}-\frac{11}{48}\right)}{12\cdot 2^{\frac34}\cdot 3^{\frac14}\cdot n^{\frac54}}I_{-\frac52}(n)
	+
	O\left(n^{-\frac74}e^{\pi\sqrt{\frac{2n}{3}}}\right).
\end{align*}
To finish the proof, we use the asymptotic expansion of the Bessel function \cite[equation (4.12.7)]{AAR}
\begin{equation*}
I_k(x) \sim \frac{e^x}{\sqrt{2\pi x}}\left(1-\frac{4k^2-1}{8x}+O\left(\frac{1}{x^2}\right)\right) \qquad \textnormal{ as } x\to \infty.
\end{equation*}
Plugging in, we find that
\begin{equation*}
u(m,n) = \frac{e^{\pi \sqrt{\frac{2n}{3}}}}{16 \sqrt{3} n} \left(1 - \frac{1}{\sqrt{n}}\left(\frac{\pi m^2}{2 \sqrt{6}} + \nu \right) + O\left(\frac{1}{n}\right)\right),
\end{equation*}
where $\nu$ is an explicit constant that does not depend on $m$ or $n$ (the value of $\nu$ is not needed to conclude the formulas in Theorem \ref{thm:unimodalAsymp}, but for the sake of the interested reader we note that $\nu = \frac{\sqrt{3}}{\sqrt{2} \pi} - \frac{11 \pi}{24 \sqrt{6}}$).

\section{Proof of Theorem \ref{thm:distributionAsymp}, Corollary \ref{C:normal}, and Corollary \ref{C:posmoments}}\label{sec:Distribution}
We use Ingham's Tauberian theorem to obtain the asymptotic main term of the rank moments.
\begin{theorem}\label{TauberianT}
Let $f(q)=\sum_{n\geq 0}a(n) q^n$ be a power series with weakly increasing non-negative
coefficients and radius of convergence equal to one. If there exist constants $A>0$ and
$\lambda, \alpha\in\R$ such that as $t \to 0^+$ we have
\[
f \left(e^{-t} \right)  \sim
\lambda   t^\alpha
e^{\frac{A}{t}},
\]
then, as $n\to\infty$,
\[
a(n) \sim\frac{\lambda}{2\sqrt{\pi}}\,
\frac{A^{\frac{\alpha}{2}+\frac14}}{n^{\frac{\alpha}{2}+\frac34}}\,
e^{2\sqrt{An}}.
\]
\end{theorem}

In order to apply Theorem \ref{TauberianT}, we  need to know that the moments of the unimodal rank are monotonic.

\begin{proposition}
\label{P:mono}
If $k,n \in \N_0$, then
\begin{equation*}
u_{2k}(n+1) \geq u_{2k}(n).
\end{equation*}
\end{proposition}
\begin{proof}
Recall that the rank moments are defined as
\begin{equation*}
u_{2k}(n) = \sum_{|\sigma| = n} \rank(\sigma)^{2k}
= \sum_{|\sigma| = n} \left|\rank(\sigma)\right|^{2k}.
\end{equation*}
There is a natural injection which we denote by $\phi$ that sends unimodal sequences of size $n$ to unimodal sequences of size $n+1$ and preserves the rank.  In particular, suppose that $\sigma$ is the unimodal sequence $\{a_1, \dots, a_m, \dots, a_s\}$ with $|\sigma| = n$ and peak $a_m$. Then set
\begin{equation*}
\phi(\sigma) := \{a_1, \dots, a_m + 1, \dots a_s\}.
\end{equation*}
It is clear that $\rank(\phi(\sigma)) = \rank(\sigma)$, and that $\phi$ is an injection (whose image contains all strongly unimodal sequences of $n+1$ whose peak is at least two larger than any part). The moments therefore satisfy
\begin{align*}
u_{2k}(n) = \sum_{|\sigma| = n} \left|\rank(\sigma)\right|^{2k}
= \sum_{|\sigma| = n} \left|\rank(\phi(\sigma))\right|^{2k}
\leq \sum_{|\sigma| = n+1} \left|\rank(\sigma)\right|^{2k} = u_{2 k}(n+1).
\end{align*}
The inequality holds because every term in the sum is non-negative.
\end{proof}

\begin{remark}
Proposition \ref{P:mono} can be modified so that it applies to the case of moments
for the partition rank and crank statistics.  If the rank of a partition is
positive, then the injection is defined by increasing the largest part by one,
and otherwise, a part of size one is added; the definition for the crank is
identical.  In all cases the magnitude of the rank or crank statistics do not
decrease (in fact, the statistic is only preserved by the injection in the case
that the crank is positive and the partition contains ones; in all other cases
the statistic changes by at least one). This would allow one to similarly use Ingham's Tauberian
theorem in order to prove the main asymptotic terms in \cite{BM12TAMS}.
However, this is not enough to prove the asymptotic inequality for rank and
crank moments that is the main result of that paper, as it requires a more
detailed asymptotic expansion.
\end{remark}
Theorem \ref{thm:distributionAsymp} follows from the asymptotic behavior of the moment generating function. For this set
$$
\mathbb{U}_{2 k}(q):= \sum_{n\geq 0} u_{2k}(n)q^n.
$$

\begin{theorem}\label{thm:U_2ell}
As $t \rightarrow 0^+$, we have
$$
\mathbb U_{2 k}\left(e^{-t}\right) \sim \frac{(2k-1)!!}{4}t^{-k}e^{\frac{\pi^2}{6t}}.
$$
\end{theorem}
\begin{proof}
For the proof we use the three-dimensional Euler-Maclaurin summation formula (see \eqref{E:EulerMac}). We start by writing
\begin{align*}
\mathbb U_{2k} (q) =  \frac{1}{(q)_\infty} \left(\delta_{k,0} V_0(q) + 2\sum_{m\geq 1} m^{2k} V_m(q)\right),
\end{align*}
where $\delta_{k,0}$ equals zero unless $k=0$, in which case we have one.
By Proposition \ref{P:Vm} we have
\begin{equation*}
V_0\left(e^{-t}\right) \sim \frac14.
\end{equation*}
Next we write
\begin{align*}
&\sum_{m\geq 1}m^{2k}V_m\left(e^{-t}\right)=t^{-k}\sum_{\delta_1,\delta_2\in\{0,1\}} (-1)^{\delta_1+\delta_2} \sum_{\b n\in\mathbb N_0^3} F\left(\sqrt{t}\left(n_1+\frac14+\frac{\delta_1}{2},n_2+\frac14+\frac{\delta_2}{2},n_3+1\right)\right),
\end{align*}
where $F(\b x):= x_3^{2k} e^{-2(x_1+\frac{x_3}{2})^2-6x_2^2-8(x_1+\frac{x_3}{2})x_2}$.
We now apply \eqref{E:EulerMac} in dimension three. Because of the weighting factor $(-1)^{\delta_1 + \delta_2}$, any term in \eqref{E:EulerMac} that does not depend on both $\delta_1$ and $\delta_2$ vanishes, leaving just two sums to consider.

The term corresponding to $\mathscr{S}=\{1,2\}$ is
\begin{equation}\label{remainfirst}
t^{-\ell-\frac12} \sum_{\delta_1,\delta_2\in\{0,1\}}(-1)^{\delta_1+\delta_2} \sum_{n_1,n_2\geq 0} \frac{B_{n_1+1}\left(\frac14+\frac{\delta_1}{2}\right)}{(n_1+1)!} \frac{B_{n_2+1}\left(\frac14+\frac{\delta_2}{2}\right)}{(n_2+1)!} t^{\frac{n_1+n_2}{2}} \int_0^\infty F^{(n_1,n_2,0)}(0,0,x_3)dx_3.
\end{equation}
The sum on $\delta_1,\delta_2$ evaluates as
\begin{multline*}
\sum_{\delta_1,\delta_2\in\{0,1\}} (-1)^{\delta_1+\delta_2} B_{n_1+1}\left(\frac14+\frac{\delta_1}{2}\right)B_{n_2+1}\left(\frac14+\frac{\delta_2}{2}\right) \\
=\begin{cases}
4 B_{n_1+1}\left(\frac14\right) B_{n_2+1}\left(\frac14\right)\quad &\textnormal{ if } n_1\equiv n_2\equiv 0 \pmod{2},\\
0\quad&\textnormal{ otherwise.}
\end{cases}
\end{multline*}
The dominant term from \eqref{remainfirst} comes from $n_1=n_2=0$ and contributes
\begin{equation*}
\frac{4}{\sqrt{t}} t^{-k} B_1\left(\frac14\right)^2 \int_0^\infty F(0,0,x_3)dx_3
=\frac{(2k-1)!!\sqrt{\pi}}{4\sqrt{2}} t^{-k-\frac12}.
\end{equation*}
The first term in \eqref{E:EulerMac} is of higher order.
Thus we get
\begin{equation*}
\mathbb U_{2k} \left(e^{-t}\right)
\sim
\frac{1}{\left(e^{-t}\right)_\infty} \left(\delta_{k,0}\frac14+2^{k-\frac32}\Gamma\left(\ell+\frac12\right)t^{-k-\frac12}\right)
\sim
\frac{2^{k - \frac32} \Gamma\left(k+\frac12\right)}{\left(e^{-t}\right)_\infty t^{k+\frac12}}.
\end{equation*}
Now
	\begin{equation*}
	\frac{1}{\left(e^{-t}\right)_\infty} \sim \sqrt{\frac{t}{2\pi}} e^{\frac{\pi^2}{6t}}.
	\end{equation*}
Combining gives the claim.
\end{proof}

Theorem \ref{thm:distributionAsymp} now follows from Theorem \ref{TauberianT}, Proposition \ref{P:mono}, and  Theorem \ref{thm:U_2ell}. The subsequent corollaries are then a straightforward consequence of the ``Method of Moments'', which uses the limiting behavior of the moments of a sequence of random variables to determine the limiting distribution. In the following key result $X$ (respectively $X_n$) is a random variable with distribution $\mu$ (resp. $\mu_n$), so that $\mu([a,b]) := \mathbf{P}(a \leq X \leq b)$.
\begin{theorem}[{\cite[Theorem 30.2]{billingsley}}]\label{thm:billingsley}
Suppose that the distribution of $X$ is determined by its moments, that moments of all orders exist for each $\{X_n\}_{n \geq 1}$, and that $\lim_{n \to \infty} E[X_n^r] = E[X^r]$ for $r \geq 1$. Then $X_n$ converges in distribution to $X$; i.e., if $f$ is bounded and continuous, then
\begin{equation*}
\lim_{n \to \infty} \int_a^b f(x) d\mu_n(x) = \int_a^b f(x) d\mu(x).
\end{equation*}
\end{theorem}

\begin{proof}[Proof of Corollary \ref{C:normal}]

For each $n$, let $\mathbf{P}_n$ denote the uniform probability distribution on $\cU(n)$, so that each unimodal sequence of size $n$ is chosen with probability $\frac{1}{u(n)}$. Now define a random variable $X_n$ on $\cU(n)$ by taking the normalized rank; in particular, if the outcome of the random selection is $\sigma \in \cU(n)$, then
\begin{equation*}
X_n = X_n(\sigma) := \frac{\rk(\sigma)}{\left(\frac{6n}{\pi^2}\right)^{\frac14}}.
\end{equation*}
Denote the corresponding distribution $X_n$ by $\mu_n$, and distribution function by $F_n$, so
\begin{equation*}
F_n(x) = \mu_n((-\infty, x]) := \mathbf{P}_n \left( \left\{ \sigma \in \cU(n) \; : \; X_n(\sigma) \leq x \right\} \right)
 = \frac{1}{u(n)} \sum_{m \leq \left(\frac{6n}{\pi^2}\right)^{\frac14} x} u(m,n).
\end{equation*}

Theorem \ref{thm:distributionAsymp} implies that
\begin{equation*}
\lim_{n \to \infty} E\left[X^{2 k}_n\right]
= \lim_{n \to \infty} \frac{1}{u(n)} \sum_{\sigma \in \cU(n)} X^{2k}_n
= \lim_{n \to \infty} \frac{u_{2k}(n)}{u(n) \left( \frac{6n}{\pi^2}\right)^{\frac{k}{2}} }= (2k - 1)!!,
\end{equation*}
and we also know by symmetry that $E[X^{2 k + 1}_n] = 0$ for $k,n\in\N_0$. As mentioned in the introduction, these limiting values are the moments for the standard normal random variable $Z$, which has the well-known distribution $\Phi$. We now apply Theorem \ref{thm:billingsley} to conclude that $X_n$ converges in distribution to $Z$. In particular, setting $f(x) = 1$ gives that
\begin{equation*}
\lim_{n \to \infty} \frac{1}{u(n)} \sum_{m \leq \left(\frac{6n}{\pi^2}\right)^{\frac14} x} u(m,n)
= \lim_{n \to \infty} \int_{-\infty}^x d \mu_n(x)
= \Phi(x),
\end{equation*}
which is precisely the statement of the corollary.
\end{proof}

\begin{proof}[Proof of Corollary \ref{C:posmoments}]
It is known \cite[Problem 21.2]{billingsley} that the absolute moments of the standard normal distribution $Z$  are given by
\begin{equation*}
E[|Z|^r] = \frac{2^{\frac{r}{2}}}{\sqrt{\pi}} \Gamma\left(\frac{r+1}{2}\right).
\end{equation*}
On the other hand, we have
\begin{equation*}
E\left[|X_n|^r\right] = \frac{1}{\left(\frac{6n}{\pi^2}\right)^\frac{r}{4}} \sum_{m \in \Z} |m|^r \frac{u(m,n)}{u(n)}
= \frac{u_r^+(n)}{u(n)\left(\frac{6n}{\pi^2}\right)^\frac{r}{4}}.
\end{equation*}
Theorem \ref{thm:billingsley} now implies that
\begin{equation*}
\lim_{n \to \infty} E[ |X_n|^r ] = E[|Z|^r],
\end{equation*}
which is the claimed result.
\end{proof}

\section{(Generalized) quantum modular properties}\label{sec:Modularity}

Due to equation \eqref{E:UAppell}, the function $U(w;q)$ can be recognized as
essentially a mock Jacobi form \cite{BR}. Furthermore,
Bryson, Ono, Pitman, and the fourth author \cite{BOPR} found
$U(-1;q)$ to be a so-called
quantum modular form.
Roughly speaking Zagier \cite{Za} defined \textit{quantum modular forms} to be functions
$f:\mathcal Q\to \C$ ($\mathcal Q\subset \Q$) such that for
$M=\left(\begin{smallmatrix}a&b\\c&d\end{smallmatrix}\right)\in\Gamma$
($\Gamma\le \SL_2(\Z)$) and $\chi$ a certain multiplier,
the \textit{error to modularity} given by
\begin{equation*}
f(\tau) - \chi(M)^{-1} (c\tau+d)^{-k} f(M\tau)
\end{equation*}
can be extended to an open subset of $\R$ as a real-analytic function. The result of \cite{BOPR} follows by establishing that $U(-1;q)$
is dual to Kontsevich's ``strange'' function, $F(q):=\sum_{n\ge0}(q)_n$,
in that $U(-1;q)=F(q^{-1})$ when $q$ is a root of unity. Note that $F(q)$ converges on no open subset of $\mathbb{C}$, and does not give a well defined series in $q$.

This has been generalized in several different ways. In a paper with Folsom, two of the authors \cite{BFR15} showed that certain weighted, twisted moments of the strong unimodal rank are quantum modular forms.
Furthermore, using $U(w,q)$ as a prototypical example, Folsom and the first author \cite{BF} introduced the notion of {\it quantum Jacobi forms}, which are functions defined on subsets of $\Q \times \Q$ such that the ``errors'' to both modular and elliptic transformations are well-behaved (as real-analytic functions).
Folsom, Ki, Vu, and Yang \cite{FKVY} found that $U(w;q)$ demonstrates
quantum modular behavior for general $w$  and is dual to a suitable
two variable analog of $F(q)$. In Hikami and Lovejoy \cite{HL}
considered multi-sum versions of $U(q)$ and $F(q)$, and
established both duality and quantum modularity.  With these results for $U(w;q)$ in mind,
it is then reasonable to ask if the functions $U_m(q)$ have any
modularity properties. While it is likely too much to ask for them to be mock
modular forms, some sort of quantum modular properties are not an
unreasonable expectation.

We close this paper by describing a potential analytic framework for understanding the modularity of the unimodal rank generating functions $U_m(q)$. First, we recall another important example of a real analytic modular form associated to a combinatorial $q$-series.
Andrews, Dyson, and Hickerson \cite{ADH} defined
\begin{equation*}
\sigma(q):=\sum_{n\geq 0} \frac{q^{\frac{n(n+1)}{2}}}{(-q)_n},\qquad \sigma^*(q) := 2\sum_{n\geq 1} \frac{(-1)^nq^{n^2}}{\left(q;q^2\right)_n}.
\end{equation*}
Note \cite{ADH, Zw2}
that $\sigma$ and $\sigma^*$ can be written as indefinite theta functions
\begin{align*}
q^{\frac{1}{24}} \sigma(q) &= \left(\sum_{\substack{n+j\geq 0 \\ n-j \geq 0}} + \sum_{\substack{n+j< 0 \\ n-j < 0}}\right) (-1)^{n+j} q^{\frac32\left(n+\frac16\right)^2-j^2},\\
q^{-\frac{1}{24}} \sigma^*(q) &= \left(\sum_{\substack{2j+3n\geq 0 \\ 2j-3n>0}} + \sum_{\substack{2j+3n < 0 \\ 2j-3n\leq 0}}\right)(-1)^{n+j} q^{-\frac32\left(n+\frac16\right)^2 +j^2}.
\end{align*}
Cohen \cite{Co} then viewed these functions in the framework of Maass forms. To recall his results, define the coefficients $T(n)$ by
\begin{equation*}
\sum_{\substack{n\in\Z \\ n\equiv 1\pmod{24}}} T(n) q^{\frac{|n|}{24}} := q^{\frac{1}{24}}\sigma(q)+ q^{-\frac{1}{24}}\sigma^*(q)
\end{equation*}
and set, $\tau = u+iv$,
\begin{equation*}
\varphi_0(\tau) := v^{\frac12} \sum_{n\in \Z\setminus\{0\}} T(n) K_0\left(\frac{2\pi|n| v}{24}\right) e^{\frac{2\pi inu}{24}},
\end{equation*}
where $K_0$ is the $K$-Bessel function of weight zero of the second kind.
Cohen then proved that $\varphi_0$ is a Maass form of weight zero on $\Gamma_0(2)$ (with some multiplier) and eigenvalue $\frac14$. Maass forms transform like modular forms. However, instead of being meromorphic they are eigenfunctions under the Laplace operator
\[
\Delta:=-v^2\left(\frac{\partial^2}{\partial u^2}+\frac{\partial^2}{\partial v^2}\right).
\]
The connection to Maass forms directly gives that $\sigma$ is a quantum modular form.

Zwegers \cite{Zw2} considered general indefinite theta series of the shape
satisfied by $\sigma$ and $\sigma^*$ and associated functions similar to that of
$\phi_0$. For convenience, suppose that $\Phi^+$ and $\Phi^-$ are the indefinite
theta series and the associated function is $\Phi$. By construction, $\Phi$ is harmonic, but $\Phi$ may not have any modular properties.
However, Zwegers was able to ``complete'' $\Phi$ to a function $\widehat\Phi$ that
satisfies a modular transformation, but may no longer be harmonic. In the case
that $\Phi=\widehat\Phi$, the functions
$\Phi^+$ and $\Phi^-$ are quantum modular forms
due to $\Phi$ being a Maass form (see \cite[Theorem 2.8]{BrLoRo} for a precise
statement).

From equation \eqref{VmPartial}, we have an indefinite theta representation of
$V_m(q)$ and so we can apply Zwegers' machinery. However, in doing so we find that
we are in the case that the associated harmonic object is not equal to its
modular completion. For this reason, we suspect that our functions are not
quantum modular forms in the sense above.

We pose two problems. First, determine any generalized quantum
modular properties of $U_m(q)$. Second, more generally, in the
case of Zwegers' construction when $\Phi\not=\widehat\Phi$, determine any
generalized quantum modular properties of $\Phi^+$ and $\Phi^-$.


\begin{thebibliography}{99}

\bibitem{And98} G. Andrews,   \emph{The theory of partitions},
Cambridge University Press, Cambridge, 1998.

\bibitem{An11}G. Andrews.
\newblock \textit{Concave and convex compositions.}
\newblock Ramanujan J. {\bf 31} (2013), 67--82.

\bibitem{AAR} G. Andrews, R. Askey, and R. Roy, \emph{Special functions}, Encyclopedia of Mathematics and its Applications
{\bf 71} (1999), Cambridge University Press, Cambridge.

\bibitem{AB} G. Andrews and B. Berndt,
\newblock \textit{Ramanujan's lost notebook. {P}art {II}}.
\newblock Springer, New York, 2009.

\bibitem{ACK} G. Andrews, S. Chan, and B. Kim, \emph{The odd moments of ranks and crank,} J. Comb. Theory A {\bf 120} (2013), 77--91.

\bibitem{ADH} G. Andrews, F. Dyson, and D. Hickerson, \emph{Partitions and indefinite quadratic forms}, Invent. Math. {\bf 91} (1988), 391--407.

\bibitem{AG} G. Andrews and F. Garvan, {\it Dyson's crank of a partition},
Bull. Amer. Math. Soc. {\bf 18} (1988), 167-171.

\bibitem{Apo90} T. Apostol, \emph{Modular Functions and Dirichlet Series in Number Theory
Series:} Grad. Texts Math. {\bf 41},
2nd ed., 1990.

\bibitem{ASD} A. Atkin and P. Swinnerton-Dyer, {\it
Some properties of partitions}, Proc. London Math. Soc. {\bf 4}
(1954), 84--106.

\bibitem{Aul51} F. Auluck, \emph{On some new types of partitions associated with generalized Ferrers graphs}, Proc. Cambridge Phil. Soc. {\bf 47} (1951), 679--686.

\bibitem{billingsley} P. Billingsley, \emph{Probability and measure. Third edition.} Wiley Series in Probability and Mathematical Statistics. John Wiley \& Sons, Inc., New York, 1995.

\bibitem{BF} K. Bringmann and A. Folsom, {\it Quantum Jacobi forms and finite evaluations of unimodal rank generating functions}, Archiv der Mathematic \textbf{107}, special volume for E. Gekeler (2016), 367--378.

\bibitem{BFR15} K. Bringmann, A. Folsom, and R. Rhoades, \emph{Unimodal sequences and ``strange'' functions: a family of quantum modular forms}, Pacific J. Math. {\bf 274} (2015), 1--25.

\bibitem{BrLoRo} K.~{Bringmann}, J.~{Lovejoy}, and L.~{Rolen},
\newblock {\it On some special families of $q$-hypergeometric Maass forms},
\newblock Int. Math. Res. Not. \textbf{18} (2018), 5537--5561.

\bibitem{BM} K. Bringmann and K. Mahlburg, \emph{An extension of the Hardy-Ramanujan Circle Method and  applications to partitions without sequences},
Amer. J. Math  \textbf{133}  (2011),  1151--1178.

\bibitem{BM12TAMS} K.~Bringmann and K.~Mahlburg,
\newblock {\em Asymptotic inequalities for positive crank and rank moments},
\newblock { Trans. Amer. Math. Soc.} \textbf{366} (2014), 1073--1094.

\bibitem{BM14JCTA}
K. Bringmann and K. Mahlburg, \emph{Asymptotic formulas for stacks and unimodal sequences}, J. Comb. Theory A {\bf 126} (2014), 194--215.

\bibitem{BMR11} K. Bringmann, K. Mahlburg, and R. Rhoades, \emph{Asymptotics for crank and rank moments}, Bull. of the London Math. Soc. {\bf 43} (2011), 661--672.

\bibitem{BO10} K. Bringmann and K. Ono, \emph{Dyson's ranks and Maass forms}, Ann. of Math. (2) {\bf 171} (2010), 419--449.

\bibitem{BR} K. Bringmann and O. Richter, {\em Zagier-type duality and lifting maps for harmonic Maass-Jacobi forms}, Adv. Math. {\bf 225} (2010), 2298--2315.

\bibitem{BOPR}
J.~Bryson, K.~Ono, S.~Pitman, and R. Rhoades,
\newblock {\em Unimodal sequences and quantum and mock modular forms},
\newblock { Proc. Natl. Acad. Sci. USA} \textbf{109} (2012), 16063--16067.

\bibitem{CM} S. Chan and R. Mao, \emph{Inequalities for ranks of partitions and the first moment of ranks and cranks of partitions}, Adv. Math. {\bf 258} (2014), 414--437.

\bibitem{Choi} Y. Choi, \emph{The basic bilateral hypergeometric series and the mock theta functions,} Ramanujan J. \textbf{24} (2011), 345--386.

\bibitem{Co} H. Cohen, \emph{$q$-identities for Maass wave forms}, Invent. Math. {\bf 91} (1988), 409--422.

\bibitem{DP15} S. DeSalvo and I. Pak, \emph{Log-concavity of the partition function}, Ramanujan J. {\bf 38} (2015), 61--73.

\bibitem{DJR13} P. Diaconis, S. Janson, and R. Rhoades, \emph{Note on a partition limit theorem for rank and crank}, Bull. Lond. Math. Soc. {\bf 45} (2013), 551--553.

\bibitem{Dys} F. Dyson, {\it Some guesses in the theory of partitions,}
Eureka (Cambridge) {\bf 8} (1944), 10--15.

\bibitem{EL41} P. Erdos and J. Lehner, \emph{The distribution of the number of summands in the partitions of a positive integer}, Duke Math. J. {\bf 8} (1941), 335--345.

\bibitem{FKVY} A. Folsom, C. Ki, Y. Vu, and B. Yang, {\it Strange combinatorial quantum modular forms}, J. Number Theory \textbf{170} (2017), 315--346.

\bibitem{FOR} A. Folsom, K. Ono, and Rob Rhoades,
\newblock {\em Mock theta functions and quantum modular forms},
\newblock {Forum of Mathematics Pi} {\bf 1} (2013), e2.

\bibitem{Ga} F. Garvan,
\newblock {\em New combinatorial interpretations of {R}amanujan's partition
  congruences mod {$5,7$} and {$11$},}
\newblock {Trans. Amer. Math. Soc.} {\bf 305} (1988), 47--77.

\bibitem{Ga10} F. Garvan, \emph{Congruences for Andrews' smallest parts partition function and new congruences for Dyson's rank},
Int. J. Number Theory {\bf 6} (2010), 281--309.

\bibitem{Ga11} F. Garvan, \emph{Higher order spt-functions},
Adv. Math. {\bf 228} (2011), 241--265.

\bibitem{HR} G. Hardy and S. Ramanujan, \emph{Asymptotic Formulaae in Combinatory Analysis}, Proc. London Math. Soc. (2) {\bf 17} (1918), 75--115.

\bibitem{HL} K. Hikami and J. Lovejoy, {\it Torus knots and quantum modular forms}, Res. Math. Sci. \textbf{2} (2015).

\bibitem{Ka} S. Kaavya,
\newblock {\em Crank 0 partitions and the parity of the partition function},
\newblock Int. J. Number Theory {\bf 7} (2011), 793--801.


\bibitem{KKS} B. Kim, E. Kim, and J. Seo, \emph{Asymptotics for $q$-expansions involving partial theta functions},
Discrete Math. {\bf 338} (2015), 180--189.

\bibitem{Mah05} K. Mahlburg, {\it Partition congruences and the
Andrews-Garvan-Dyson crank}, Proc. Natl. Acad. Sci. USA {\bf 102} (2005), 15373--15376.

\bibitem{R} S. Ramanujan, {\it Congruence properties of partitions},
Math. Z. {\bf 9} (1921), 147--153.

\bibitem{rhoades} R.~Rhoades,
\newblock {\em Asymptotics for the number of strongly unimodal sequences},
\newblock {Int. Math. Res. Not.} (2014), 700--719.

\bibitem{Stan89} R. Stanley, \emph{Log-concave and unimodal sequences in algebra, combinatorics, and geometry}, Graph theory and its applications: East and West (Jinan, 1986), 500--535, Ann. New York Acad. Sci. {\bf 576}, New York Acad. Sci., New York, 1989.


\bibitem{szekeres} G. Szekeres, \emph{Asymptotic distribution of the number and size of parts in unequal partitions}, Bull. Austral. Math. Soc. {\bf 36} (1987), 89--97.

\bibitem{Wri33} E. Wright, \emph{Asymptotic partition formulae II. Weighted partitions}, Proc. London Math. Soc. {\bf 36} (1933), 117--141.

\bibitem{Wri68} E. Wright, \emph{Stacks}, Q. J. Math. {\bf 19} (1968), 313--320.

\bibitem{Wri71} E. Wright, \emph{Stacks. II}. Q. J. Math. {\bf 22} (1971), 107--116.

\bibitem{Zag09} D. Zagier, \emph{Ramanujan's mock theta functions and their
 applications [d'apr\'es  Zwegers and Bringmann-Ono] }
Ast\'erisque {\bf 326} (2009), Soc. Math. de France, 143--164.

\bibitem{Za}D.~Zagier,
\newblock {\em Quantum modular forms,}
\newblock In {Quanta of maths}, volume~11 of {\em Clay Math. Proc.},
  659--675. Amer. Math. Soc., Providence, RI, 2010.

\bibitem{ZwegersPhD} S. Zwegers, {\it Mock theta functions}, Ph.D. Thesis, Universiteit Utrecht, 2002.

\bibitem{Zw2} S. Zwegers,
\newblock {\em Mock {M}aass theta functions},
\newblock { Q. J. Math.} {\bf 63} (2012), 753--770.

\end{thebibliography}
\end{document}